\documentclass[10pt,reqno]{amsart}   
\usepackage{amssymb,amscd,latexsym}   
\usepackage{amsmath}
\usepackage{amsthm,times}
\usepackage[all]{xy}
\usepackage{epsfig,graphicx}
\usepackage{graphicx,color}
\usepackage[utf8]{inputenc}
\usepackage{float}
\usepackage{tikz}
\usepackage{enumitem}
\usepackage{tikz-cd}
\usepackage{mathrsfs}
\newcommand{\rar}{\rightarrow}

\newcommand{\llar}{-\kern-5pt-\kern-5pt\longrightarrow}

\newcommand{\Hilb}{\text{Hilb}}
\usepackage{cases}

\usepackage{multirow}
\usepackage{makecell} 

\newtheorem{Theorem}{Theorem} 
\newtheorem{Lemma}[Theorem]{Lemma}
\newtheorem{cor}[Theorem]{Corollary}
\newtheorem{prop}[Theorem]{Proposition}
\newtheorem{Remark}[Theorem]{Remark}
\newtheorem{Example}[Theorem]{Example}

\newtheorem{defi}[Theorem]{Definition}

\DeclareMathOperator{\Spec}{Spec}

\renewcommand{\P}{\mbox{P}}
\newcommand{\CC}{\mathcal{C}}
\newcommand{\OO}{\mathcal{O}}
\newcommand{\h}{\mathcal{H}}

\def\codim{{\rm codim}\,}

\def\restr{{\kern-1pt\restriction\kern-1pt}}

\def\P{{\mathbb P}}

\def\Z{{\mathbb Z}}

\keywords{Complete intersection, Hilbert Scheme, moduli space, biprojective space.}

\subjclass[2010]{14H22, 14H10,14C05, 14D20}

\begin{document}

\title{On Hilbert scheme of complete intersection on the biprojective}

\author{Aislan Leal Fontes}
\address{Departamento de Matem\'atica, UFS - Campus Itabaiana. Av. Vereador Ol\'impio Grande s/n, 49506-036 Itabaiana/SE, Brazil}
\email{aislan@ufs.br}
\author{Maxwell da Paixão de Jesus Santos}
\address{Departamento de Matem\'atica, UFS - Campus Itabaiana. Av. Vereador Ol\'impio Grande s/n, 49506-036 Itabaiana/SE, Brazil}
\email{maxwell.agro70@yahoo.com.br}

\maketitle

\begin{abstract}
The goal of this paper is to construct the Hilbert scheme of complete intersections in the biprojective 
space \( X = \mathbb{P}^m \times \mathbb{P}^n \). For this purpose, we define a partial order on the 
bidegrees of the bihomogeneous forms. As a consequence of this construction, we provide an explicit computation of the Hilbert 
scheme for curves of genus 7 and 8 listed in \cite{MUK95} and \cite{MUKIDE03} that are complete 
intersections. Finally, we construct the coarse moduli space of canonical complete intersections in \( \mathbb{P}^m \times \mathbb{P}^n \).
\end{abstract}


\tableofcontents
\section{Introduction}

We recall that a variety $Y \subset \P^n$ is a \textbf{complete intersection in $\P^n$} if $Y$ is generated by the homogeneous polynomials $F_1, \dots, F_c$ and satisfies $0 < \dim Y = n-c$. We can generalize this definition as follows: a subvariety $Y \subset X$ is a \textbf{complete intersection} in $X$ if it is the intersection of $c = \codim(Y, X)$ hypersurfaces in $X$, where $X$ is a smooth, arithmetically Cohen--Macaulay projective variety.  
We are particularly interested in complete intersections in biprojective spaces $X = \P^{m} \times \P^{n}$.

In \cite{MUK95} and \cite{MUKIDE03}, Mukai introduced an interesting stratification 
for $\mathcal{M}_7$ and $\mathcal{M}_8$, where in both cases some strata are complete intersections 
in biprojective spaces (see Example~\ref{ex13}); this example serves as the starting point 
for our study of complete intersections with certain properties.

The second section examines the cohomology of complete intersections. A key distinction between the standard and biprojective cases is that, unlike $\mathbb{P}^n$, not all complete intersections in $X$ are arithmetically Cohen-Macaulay (ACM) as for example, the curve of type $(2,0)$ in $\mathbb{P}^1\times\mathbb{P}^1$. We say that a set of bi-homogeneous polynomials $F_1, \dots, F_c$ forms a \emph{regular sequence} if, for every subset $F_{i_1}, \dots, F_{i_\alpha}$, the corresponding subscheme is both a complete intersection and ACM under the Segre embedding. Throughout, we assume that the complete intersection $Y \subset X$ is defined by such a regular sequence; when no ambiguity arises, we simply say $Y$ is ACM.  A deeper analysis of ACM complete intersections is reserved for Section~3; here, we establish foundational results on their Hilbert polynomials and flat families, such properties depend crucially on the ACM condition.


In Section~4, our goal is to understand the Hilbert scheme $\mathcal{H}$ of complete intersections with a fixed order in~$X$. To achieve this, we apply the tools of~\cite{OB12}, where the Hilbert scheme is constructed as a tower of Grassmannians. Unlike in projective space, in biprojective space~$X$, we cannot define a total order on the bidegrees of bihomogeneous forms but only a partial order. Consequently, we adapt the theorems for standard complete intersections in~$X$. Additionally, we compute the dimension of the Hilbert scheme for the complete intersections listed by Mukai in~\cite{MUK95} and~\cite{MUKIDE03}.

Finally, we apply our results to study the moduli space $\mathcal{M}$ of canonical complete intersections in the biprojective space, more precisely, we construct the scheme $\mathcal{M}$ as the quotient of~$\mathcal{H}$ by the automorphism group of~$X$. As an application, we explicitly compute the dimension of $\mathcal{M}$ for the cases of complete intersections classified by Mukai in~\cite[Table~1]{MUK95} and ~\cite{MUKIDE03}. The advantage of our approach lies in providing a direct and explicity calculation of the dimension of $\mathcal{M}$.

\section{On the Cohomology of Complete Intersection in $\P^{m}\times \P^{n}$}
Let \( R := k[x_0, \dots, x_m, y_0, \dots, y_n] \) be the polynomial ring with coefficients in \( k \). A monomial  
\( G = x_0^{a_0}\cdots x_m^{a_m}y_0^{b_0}\cdots y_n^{b_n} \in R \) has bidegree \( (\sum a_i, \sum b_i) \). Let \( R_{i,j} \) be the  
\( r \)-dimensional vector space over \( k \) spanned by all monomials of bidegree \( (i,j) \), where  
\( r = \dim R_{i,j} = \binom{m + i}{m} \binom{n + j}{n} \). A polynomial \( F \in R \) is  
bihomogeneous of bidegree \( (d_1,d_2) \) if \( F \) is a \( k \)-linear combination of monomials of bidegree  
\( (d_1,d_2) \). We also say that \( F \) is a form of bidegree \( (d_1,d_2) \).

Let us consider the variety $X = \P^{m} \times \P^{n}$, where the inclusion $X \subset \P^{r}$ is the Segre embedding. Throughout this section, we assume $m \leq n$ and $r = (m+1)(n+1) - 1$.

\begin{defi}
\textup{Let $D_i = \OO_X(a_i, b_i)$ be effective divisors on $X$ and take $F_i \in H^{0}(X, D_i) \setminus \{0\}$, where $a_i + b_i > 1$. We say that $Y = \bigcap_{i=1}^{c} V(F_i)$ is a complete intersection in $X$ if $\dim Y = m + n - c$, where $0 \leq c < n + m$.}  
\end{defi}

We denote $Y_{j} = \bigcap_{i=1}^{j} V(F_i)$, where $Y_0 = X$ and $Y_c = Y$. Thus, if $Y$ is a complete intersection, then each $Y_j$ is a complete intersection.

A scheme $Y \subset \P^r$ is said to be \emph{Arithmetically Cohen-Macaulay (ACM)} if its coordinate ring is a Cohen-Macaulay ring. Unlike the projective space case, not all complete intersections in $X$ are ACM. For example, let $Y = V(x_0^2)$ be the curve of bidegree $(2,0)$ in $\P^1 \times \P^1$ under the Segre embedding in $\P^3$. In fact, the coordinate ring of $Y$ is $R = k[z_0, z_1, z_2, z_3]/(z_0z_3 - z_1z_2, z_0^2, z_0z_1, z_1^2)$, where $\dim_{\text{Krull}} R = 2$ and $z_2$ is a maximal regular sequence in $R$; hence, $R$ is not Cohen-Macaulay. An equivalent definition of an ACM variety is
\[ H^{i}(Y, \OO_Y(d)) = 0 \text{ for every } 1 \leq i \leq \dim Y - 1, \text{ or } H^{i}(Y, \mathcal{I}_Y(d)) = 0, \forall \, 1 \leq i \leq \dim Y. \]

\begin{defi}
\textup{Let $Y = \bigcap_{i=1}^{c} V(F_i)$ be a complete intersection in $X$. We say that $F_1, \dots, F_c$ form a regular sequence on $X$ if for every subset of indices $i_1 < i_2 < \cdots < i_{\alpha}$, the scheme $\bigcap_{r=1}^{\alpha} V(F_{i_r})$ is ACM.}    
\end{defi}

By the Künneth formula,
\[ h^{i}(X, \OO_X(a,b)) = \sum_{r+s=i} h^{r}(\P^{m}, \OO_{\P^{m}}(a)) \cdot h^{s}(\P^{n}, \OO_{\P^{n}}(b)), \]
we summarize the cohomology of $X$ in Table \ref{T1}.

\begin{table}[h]
\begin{tabular}{|l|l|l|l|l|}
\hline
 & $h^{0}(\OO_X(a,b))$    & $h^{m}(\OO_X(a,b))$ & $h^{n}(\OO_X(a,b))$ &  $h^{m+n}( \OO_X(a,b))$ \\ \hline
$a \geq 0$, $b \geq 0$ & ${m+a \choose m}\cdot{n+b \choose n}$ & 0 & 0 & 0 \\ \hline
  $\begin{matrix}
b  \geq  0,\\
a  \leq  -m-1
\end{matrix}$ & 0 & ${n+b \choose n}\cdot{-1-a \choose m}$ & 0 & 0 \\ \hline
 $\begin{matrix}
a  \geq  0,\\
b  \leq  -n-1
\end{matrix}$ &  0 &  0 & ${m+a \choose m}\cdot{-1-b \choose n}$ & 0  \\ \hline
$\begin{matrix}
a  \leq  -m-1,\\
b  \leq  -n-1
\end{matrix}$ & 0 & 0 &  0 &  ${-1-a \choose n}\cdot{-1-b \choose n}$ \\ \hline
\end{tabular}
\caption{Cohomology of $\P^{m} \times \P^{n}$.}
\label{T1}
\end{table}
\begin{prop}{\label{K}}
Let $Y=\bigcap_{i=1}^{c}V(F_i)$ be a closed subscheme of $X$. If $(F_1,\dots,F_c)$ is a regular sequence, then:
    \begin{enumerate}
        \item For all $0\leq j \leq c$, $Y_j := \bigcap_{i=1}^j V(F_i)$ is a complete intersection in $X$ of codimension $j$;
        \item $Y$ is generated by a regular sequence;
        \item The dualizing sheaf of $Y$ is $\omega_Y \cong \OO_Y\left(\sum a_i - m-1,\sum b_i-n-1\right)$.
    \end{enumerate}
\end{prop}

\begin{proof}
Every subsequence of a regular sequence is regular; thus, each $Y_j$ is a complete intersection. The regularity condition ensures that $Y_j$ has the expected codimension $j$.

From the Koszul complex, we obtain the exact sequence:
  
\begin{equation*}\label{eq7}
    \centering
\begin{tikzcd}
 \displaystyle\bigoplus_{i<j}\OO_{X}(-a_i-a_j,-b_i-b_j) \rar & \displaystyle\bigoplus_{i}\OO_{X}(-a_i,-b_i)  \rar &  \mathcal{I}_{Y,X} \rar & 0, \\
\end{tikzcd}
\end{equation*}
where $\mathcal{I}_{Y,X}$ is the ideal sheaf of $Y$ in $X$. Therefore,
$$\omega_Y \cong \omega_X \otimes (\det \mathcal{N}_{Y/X}) \cong \OO_Y\left(\sum a_i - m-1,\sum b_i-n-1\right),$$
where $\mathcal{N}_{Y/X}$ is the normal sheaf of $Y$ in $X$.
\end{proof}

\begin{prop}{\label{P10}} Let $Y$ be a complete intersection in $\mathbb{P}^m \times \mathbb{P}^n$ such that either $m > 1$, or $m = 1$ and $a_i \leq b_i$ for all $i >1$. Then, $Y$ is generated by a regular sequence if and only if, 
$$(a_{i_1}-b_{i_1}+\cdots+a_{i_{\alpha}}-b_{i_{\alpha}})<m+1$$ 
and 
$$(b_{j_1}-a_{j_1}+\cdots+b_{{j_\beta}}-a_{j_{\beta}})< n+1$$ 
for every $i_1<i_2<\cdots<i_{\alpha}$ and $j_1<j_2<\cdots <j_{\beta}$.
\end{prop}

\begin{proof}
Let \( Y = \bigcap_{i=1}^{c} V(F_i) \) be a complete intersection in \( X \). Let us prove the statement by induction on \( c \).  

If \( c = 1 \), then the ideal sheaf of \( Y \) in \( X \) is isomorphic to \( \OO_X(-a_1, -b_1) \). For every \( d \in \mathbb{Z} \), consider the exact sequence  
\begin{equation}\label{se1}  
\centering  
\begin{tikzcd}  
0 \rar & \OO_X(d-a_1, d-b_1) \rar{\cdot F_1} & \OO_X(d, d) \rar & \OO_{Y_1}(d, d) \rar & 0,  
\end{tikzcd}  
\end{equation}  
and observe that \( F_1 \) is regular if and only if  
\[ h^{i}(X, \OO_X(d-a_1, d-b_1)) = 0 \text{ for all } 1 \leq i \leq m+n-1, \]  
which is equivalent to  
\[ h^{i}(X, \OO_X(d-a_1, d-b_1)) = 0 \text{ for all } i \in \{m, n\} \iff (a_1 - b_1) < m+1 \text{ and } (b_1 - a_1) < n+1. \]  

Now suppose \( c > 1 \). Consider the bigraded Koszul complex below:  
\begin{equation}{\label{k1}}  
\centering  
\begin{tikzcd}  
0 \rar & \OO_X\left(\sum_{i=1}^{c} (-a_i, -b_i)\right) \rar & \cdots \rar \ar[draw=none]{d}[name=X, anchor=center]{}  
& \bigoplus_{i<j} \OO_X(-a_i - a_j, -b_i - b_j) \ar[rounded corners,  
to path={ -- ([xshift=2ex]\tikztostart.east)  
|- (X.center) \tikztonodes  
-| ([xshift=-2ex]\tikztotarget.west)  
-- (\tikztotarget)}]{dll}[at end]{} \\  
& \bigoplus_{i} \OO_X(-a_i, -b_i) \rar & \mathcal{I}_{Y,X} \rar & 0,  
\end{tikzcd}  
\end{equation}  
where \( \mathcal{I}_{Y,X} \) is the ideal sheaf of \( Y \) in \( X \). Suppose  
\[  
\begin{matrix}  
(a_{i_1} - b_{i_1} + \cdots + a_{i_\alpha} - b_{i_\alpha}) & < & m+1 \\  
& \text{and} & \\  
(b_{j_1} - a_{j_1} + \cdots + b_{j_\beta} - a_{j_\beta}) & < & n+1  
\end{matrix}  
\]  
for every \( i_1 < i_2 < \cdots < i_\alpha \) and \( j_1 < j_2 < \cdots < j_\beta \).  
Then, for all \( d \in \mathbb{Z} \) and \( 1 \leq j \leq n + m - 1 \), we have  
 $$h^{j}(\OO_X(\sum_{i=1}^{c} (d - a_i, d - b_i))) = \cdots = \displaystyle\sum_{i} h^{j}(\OO_X(d - a_i, d - b_i)) = 0.$$  
Thus, \( h^{j}(\mathcal{I}_{Y,X}(d)) = 0 \) for all \( d \in \mathbb{Z} \) and \( 1 \leq j \leq n + m - c \), i.e., \( Y \) is generated by a regular sequence.

Conversely, if \( Y \) is generated by a regular sequence, then by induction, we only need to show that  
\[
\sum_{i=1}^{c}(a_i - b_i) < m + 1 \quad \text{and} \quad \sum_{i=1}^{c}(b_i - a_i) < n + 1.
\]  
Tensoring the bigraded Koszul complex \eqref{k1} by \( d \in \mathbb{Z} \), we obtain  
\[
H^{j}(\mathcal{I}_{Y,X}(d)) = H^{j+1}(\mathcal{O}_X(d - \sum_{i=1}^{c} a_i, d - \sum_{i=1}^{c} b_i))
\]  
for all \( j \in \{1, \dots, m + n - c\} \), and these cohomology groups vanish by hypothesis. The only remaining case to check is the vanishing of  
\[
H^{1}(\mathcal{O}_X(d - \sum_{i=1}^{c} a_i, d - \sum_{i=1}^{c} b_i)).
\]  
If \( m > 1 \), this follows directly from Table \ref{T1}. If \( m = 1 \), then by assumption,  
\[
\sum_{i=1}^{c}(a_i - b_i) < 2,
\]  
so the cohomology group vanishes in this case as well.
\end{proof}

It is important to note that there are many cases of complete intersections generated by a regular sequence to be studied. The simplest example is the curves of bidegree $(a,b)$ on a smooth quadric in $\mathbb{P}^3$, where $|a-b|<2$. Here are some interesting examples:

\begin{Example}{\label{ex13}}
Let $\mathcal{C}$ be a smooth curve of genus 8. From \cite[(i) and (ii) of the Theorem]{MUKIDE03}, we have:
\begin{enumerate}
    \item If $\mathcal{C}$ is a general curve and has a non-selfadjoint $g^2_7$, then $\mathcal{C}$ is a complete intersection of divisors of bidegree $(1,1)$, $(1,2)$, and $(2,1)$ in $\mathbb{P}^{2} \times \mathbb{P}^{2}$;
    \item If $\mathcal{C}$ has a $g^1_4$ but no $g^2_6$, then $\mathcal{C}$ is the complete intersection of four divisors of bidegree $(1,1)$, $(1,1)$, $(0,2)$, and $(1,2)$ in $\mathbb{P}^{1} \times \mathbb{P}^{4}$.
\end{enumerate}

Let $\mathcal{C}$ be a smooth curve of genus 7. From \cite[Table 1]{MUK95}, we have:
\begin{enumerate}
    \item If $\mathcal{C}$ is trigonal and has a non-selfadjoint $g^2_6$, then $\mathcal{C}$ is a complete intersection of two divisors of bidegree $(1,1)$ and $(3,3)$ in $\mathbb{P}^1 \times \mathbb{P}^2$;
    \item If $\mathcal{C}$ is tetragonal and has no $g^2_6$, then $\mathcal{C}$ is isomorphic to a complete intersection of a divisor of bidegree $(1,1)$ and two divisors of bidegree $(1,2)$ in $\mathbb{P}^1 \times \mathbb{P}^3$;
    \item If $\mathcal{C}$ has gonality 4 and a non-selfadjoint $g^2_6$, then $\mathcal{C}$ is isomorphic to a complete intersection of three divisors of bidegree $(1,1)$, $(1,1)$, and $(2,2)$ in $\mathbb{P}^2 \times \mathbb{P}^2$.
\end{enumerate}
    \end{Example}


   \begin{cor}\label{cor11}
    Under the hypotheses of Proposition~\ref{P10}, the following hold:
    \begin{enumerate}[label=\textbf{C.\arabic*}]
        \item \label{itema1} The map $H^{0}(X, \OO_X(d)) \longrightarrow H^{0}(Y, \OO_Y(d))$ is surjective for every $d \in \Z$;
        \item \label{itema2} The map $H^{0}(\P^{r}, \OO_{\P^{r}}(d)) \longrightarrow H^{0}(Y, \OO_Y(d))$ is surjective for every $d \in \Z$;
        \item \label{itema3} The kernel $H^{0}(X, \mathcal{I}_{Y,X}(d))$ consists of bi-homogeneous polynomials of the form $F = \sum_{i} F_i H_i$, where $H_i \in H^{0}(X, \OO_X(d - a_i, d - b_i))$.
    \end{enumerate}
\end{cor}

    \begin{proof}
       Items~\ref{itema1} and~\ref{itema3} follow from Proposition~\ref{P10}. Let us now consider the exact sequence
\begin{equation*}  
\centering  
\begin{tikzcd}  
0 \rar & \mathcal{I}_{X} \rar & \OO_{\P^{r}} \rar & \OO_{X} \rar & 0,  
\end{tikzcd}  
\end{equation*}  
where the map $H^{0}(\P^{r}, \OO_{\P^{r}}(d)) \longrightarrow H^{0}(X, \OO_X(d))$ is given by sending a homogeneous monomial $Z_{00}^{i_{00}}Z_{01}^{i_{01}}\cdots Z_{mn}^{i_{mn}}$ of degree $d$ to a form $(X_0Y_0)^{i_{00}}(X_0Y_1)^{i_{01}}\cdots (X_mY_m)^{i_{mn}}$ of bidegree $(d,d)$. It is easy to see that the map is surjective; now, applying item \ref{itema1}, we complete the proof.
    \end{proof}

\begin{cor}{\label{cor12}}
  Let $Y$ be a complete intersection on $X \subset \mathbb{P}^{r}$. If $Y$ is generated by a regular sequence, then:
\begin{enumerate}
    \item[(a)] The scheme $Y$ is connected.
    \item[(b)] $Y \subset \mathbb{P}^{r}$ is degenerate if and only if $(a_i, b_i) = (1,1)$ for some $i$.
    \item[(c)] $Y$ is not contained in a global section of either $\mathcal{O}_X(1,0)$ or $\mathcal{O}_X(0,1)$.
\end{enumerate}
\end{cor}
\begin{proof}
By considering the exact sequence
\begin{equation*}
    \centering
    \begin{tikzcd}
        0 \rar & \mathcal{I}_{Y,X} \rar & \OO_{X} \rar & \OO_{Y} \rar & 0,
    \end{tikzcd}
\end{equation*}
it follows from \ref{itema1} and \ref{itema3} for $d=0$ that $k=H^{0}(X, \OO_X) \cong H^{0}(Y, \OO_Y)$, which means $Y$ is connected. The second item simply takes $d=1$.

For the third item, we consider the exact sequence
\begin{equation*}
    \centering
    \begin{tikzcd}
        0 \rar & \OO_{Y_{j-1}}(-a_j, -b_j) \rar & \OO_{Y_{j-1}} \rar & \OO_{Y_j} \rar & 0,
    \end{tikzcd}
\end{equation*}
and tensor it by $\OO_X(1,0)$ or $\OO_X(0,1)$. For all $j \geq 1$, we will show that
\[
H^{0}(Y_{j-1}, \OO_{Y_{j-1}}(1-a_j, -b_j)) = H^{0}(Y_{j-1}, \OO_{Y_{j-1}}(-a_j, 1-b_j)) = 0,
\]
or equivalently,
\[
h^q(X, \mathcal{F} \otimes \mathcal{O}_X(1,0)) = h^q(X, \mathcal{F} \otimes \mathcal{O}_X(0,1)) = 0,
\]
for all $q \in \{0, m, n\}$ and all $i_1 < \cdots < i_{m+1}$, where
\[
\mathcal{F} = \OO_X(-a_{i_1} - \cdots - a_{i_{m+1}}, -b_{i_1} - \cdots - b_{i_{m+1}}).
\]

Since $a_i + b_i > 1$, the case $q = 0$ is trivial. Let us now consider the case $q = m$. Without loss of generality, we may assume $m < n$. Take indices $i_1 < \cdots < i_m < i_{m+1}$. A necessary condition for
\[
h^{m}(X, \mathcal{F} \otimes \mathcal{O}_X(1,0))
\]
to be nonzero is $b_{i_1} = \cdots = b_{i_{m+1}} = 0$. In particular, $a_{i_p} \geq 2$ for all $p \in \{1, \dots, m, m+1\}$, which implies
\[
2m \leq a_{i_1} + \cdots + a_{i_m} < m + 1,
\]
a contradiction.

Similarly, a necessary condition for
$h^{m}(X, \mathcal{F} \otimes \mathcal{O}_X(0,1))$ to be nonzero is $b_{i_1} + \cdots + b_{i_{m+1}} \leq 1$. We may suppose 
$b_{i_{m+1}} = 1$ and $b_{i_1} = \cdots = b_{i_m} = 0$, which in turn implies $a_{i_p} \geq 2$ for all 
$p \in \{1, \dots, m\}$. Following the same reasoning as before, we obtain a contradiction. The case for the $n$-th cohomology vanishes analogously.

\end{proof}
\section{Complete  Intersection Arithmetically Cohen-Macaulay on $\P^m\times \P^n$}\label{sec4}

In the study of the standard complete intersection in $\P^{n}$ we always order the degrees of the homogeneous forms such 
that $d_{i} \leq d_{i+1}$. Unfortunately, we cannot get a total order for the bidegrees in $X$. We need a partial 
order such that the intermediate cohomologies vanish to replicate the theorems of standard complete intersections, but 
now in $X$. That is, there is an ordering $Y = \cap_{i=1}^{c}V(F_i)$, where the map 
\[
\varphi_j:H^{0}(X, \OO_X(a_{j},b_{j})) \longrightarrow H^{0}(Y_{j-1}, \OO_{Y_{j-1}}(a_{j},b_{j}))
\] 
is surjective for every $j$. To achieve this, we only need to show that  
\begin{equation}  
\begin{matrix}
h^{m}(X, \OO_{X}(a_{i_{\gamma}}-a_{i_{\alpha}}-\cdots-a_{i_{1}},b_{i_{\gamma}}-b_{i_{\alpha}}-\cdots-b_{i_{1}})) & = & 0 \\  
\text{and} \\  
h^{n}(X, \OO_{X}(a_{i_{\gamma}}-a_{i_{\alpha}}-\cdots-a_{i_{1}},b_{i_{\gamma}}-b_{i_{\alpha}}-\cdots-b_{i_{1}})) & = & 0,
\end{matrix}  
\label{eq5}
\end{equation}  
for every $i_1 < \cdots < i_m$ and for every $\gamma \not\in\{1, \cdots,\alpha\}$.  This is equivalent to the remark:
\begin{Remark}{\label{ord}}  
For every \( i_1 < \cdots < i_\alpha \) and for every \( \gamma \notin \{1, \dots, \alpha\} \), we have  
\begin{enumerate}  
    \item \(\sum_{j=1}^{\alpha} a_{i_j} < a_{i_\gamma} + m + 1\) or \( b_{i_\gamma} < \sum_{j=1}^{\alpha} b_{i_j} \);  
    \item \(\sum_{j=1}^{\alpha} b_{i_j} < b_{i_\gamma} + n + 1\) or \( a_{i_\gamma} < \sum_{j=1}^{\alpha} a_{i_j} \).  
\end{enumerate}  
\end{Remark}

Many complete intersections satisfy this relation. Indeed, every curve in Example~\ref{ex13} satisfies the 
conditions of Remark~\ref{ord}. 

We say a scheme $Y$ is a \emph{complete intersection ACM} (or simply \emph{ACM}) if it is defined by a regular sequence and its bidegree satisfies the criterion in Remark~\ref{ord}. This definition generalizes the one given by E.~Ballico in~\ref{EB22}. From now on, we are fixing the lexicographic order on the bidegrees of bihomogeneous forms. 

Let us consider
$(\alpha_1, \beta_1), \dots, (\alpha_s, \beta_s)$ be the distinct bidegrees 
among the pairs $(a_i, b_i)$, grouped as follows:

\begin{equation}
    \begin{matrix}{\label{o1}}
(a_1, b_1) &=& \cdots &= & (a_{m_1}, b_{m_1}) &=& (\alpha_1, \beta_1), \\
(a_{m_1 + 1}, b_{m_1 + 1}) &=& \cdots &= & (a_{m_1 + m_2}, b_{m_1 + m_2}) &=& (\alpha_2, \beta_2), \\
\vdots && \cdots && \vdots && \vdots \\
(a_{\sum_{i=1}^{s-1}m_i + 1}, b_{\sum_{i=1}^{s-1}m_i + 1}) &=& \cdots &= & (a_{\sum_{i=1}^{s}m_i}, b_{\sum_{i=1}^{s}m_i}) &=& (\alpha_s, \beta_s).
\end{matrix}
\end{equation}
Thus, when we say that \( Y \) is a \emph{complete intersection} of bidegrees \( (\alpha_1, \beta_1), \dots, (\alpha_s, \beta_s) \), we mean that its defining ideal is generated by \( m_i \) forms of bidegree \( (\alpha_i, \beta_i) \) for each \( 1 \leq i \leq s \). Explicitly, \( Y =V(F_1,\cdots,F_{m_1},F_{m_1+1},\cdots,F_{m_1+m_2},\cdots,F_s)\) .

\begin{prop}\label{prop15}
Let \( Y = \bigcap_{i=1}^{c} V(F_i) \) be an complete intersection ACM in \( X \). Then the following hold:
\begin{enumerate}
    \item For each \( j \), the restriction map 
    \[
    \varphi_j \colon H^{0}\big(X, \OO_X(a_j, b_j)\big) \to H^{0}\big(Y_{j-1}, \OO_{Y_{j-1}}(a_j, b_j)\big)
    \]
    is surjective, where \( Y_j = \bigcap_{i=1}^j V(F_i) \) and \( Y_0 = X \).

    \item The kernel of \( \varphi_j \), given by \( H^{0}\big(X, \mathcal{I}_{Y_{j-1}, X}(a_j, b_j)\big) \), consists of bihomogeneous polynomials of the form 
    \[
     F=\sum_{i} F_i G_i,
    \]
    where \( G_i \in H^{0}\big(X, \OO_X(a_j - a_i, b_j - b_i)\big) \).

    \item The kernel \( H^{0}\big(X, \mathcal{I}_{Y, X}(d)\big) \) depends only on the bidegrees \( (a_1, b_1), \dots, (a_c, b_c) \) and the integers \( d, c, m, n \).

    \item For every \( 1 \leq r \leq s \), the scheme \( Y \) is contained in a unique complete intersection \( X_{r-1} \) of bidegrees \( (\alpha_1, \beta_1), \dots, (\alpha_{r-1}, \beta_{r-1}) \), where \( X_{r-1} = V(F_{1}, \dots, F_{\mu_{r-1}}) \) and $\mu_{r-1}=m_1+\cdots+m_{r-1}$.
\end{enumerate}
\end{prop}

\begin{proof}
For the first item, the surjectivity of \( \varphi_j \) follows immediately from the ACM property of \( Y \). To describe the kernel, we tensor the Koszul complex (see \ref{k1}) on \( Y_{j-1} \) with \( \OO_X(a_j, b_j) \), which yields the claimed decomposition. 

For the third item of this proposition, we consider the short exact sequence
\[
\begin{tikzcd}
0 \rar & \OO_{Y_{j-1}}(d - a_j, d - b_j) \rar & \OO_{Y_{j-1}}(d) \rar & \OO_{Y_j}(d) \rar & 0.
\end{tikzcd}
\]
The conclusion follows by induction on \( j \).

Finally, we prove the last item.  Suppose \( X_{r-1}' \) is another complete intersection of the same bidegrees containing \( Y \). By item (2) of that proposition, the defining equations of \( X_{r-1}' \) are polynomials combinations of \( F_1, \dots, F_{\mu_{r-1}} \), so \( X_{r-1}' \subseteq X_{r-1} \). Since both have the same Hilbert polynomial (by item (3)), equality holds: \( X_{r-1}' = X_{r-1} \). 
\end{proof}

Let $Y$ be a complete intersection subscheme of $X \subset \P^r$. For a coherent sheaf $\mathcal{F}$ on $\P^r$, the Hilbert polynomial $p_{\mathcal{F}}$ is given by
\[
p_{\mathcal{F}}(d) = \sum_{i \geq 0} (-1)^i h^i(\P^r, \mathcal{F}(d)),
\]
and the Hilbert polynomial $p_Y$ of $Y$ is defined as that of its structure sheaf $\OO_Y$.

Let $\mathcal{I}_Y$ be the ideal sheaf of $Y$ in $\P^r$ and $\mathcal{I}_{Y,X} = \mathcal{I}_Y|_X$ its 
restriction to $X$. By Serre's Vanishing Theorem, there exists an integer $d_0$ such that for all $d \geq d_0$ and $i > 0$,
\[
h^i(\P^r, \OO_Y(d)) = 0.
\]
Thus, for $d \gg 0$, the Hilbert polynomial $p_Y(d)$ is determined only by $h^0(\P^r, \OO_Y(d))$.

Consider the short exact sequence of sheaves on $X$:
\[
\begin{tikzcd}
0 \arrow{r} & \mathcal{I}_{Y,X} \arrow{r} & \OO_X \arrow{r} & \OO_Y \arrow{r} & 0.
\end{tikzcd}
\]
Twisting it by $\OO_X(d)$ and taking global sections, we obtain, for $d$ sufficiently large,
\[
p_Y(d) = h^0(X, \OO_X(d)) - h^0(X, \mathcal{I}_{Y,X}(d)).
\]
From Proposition \ref{prop15}, for $d \gg 0$, the dimension $h^0(X, \mathcal{I}_{Y,X}(d))$ depends only on $m$, $n$, $d$, $c$, and the bidegrees of the defining equations. In particular, if $Y = \bigcap_{i=1}^c V(F_i)$ and $Y' = \bigcap_{i=1}^c V(F_i')$, where $F_i, F_i' \in H^0(X, \OO_X(a_i, b_i))$, then $p_Y = p_{Y'}$.\\
\begin{prop}\label{} Let $Y= V(F)$ be a hypersurfaces of $X$ of bidegree $(a,b)$. If $Y'$ is variety  of $X$ with $p_{Y'}= p_Y$, then $Y'$ is a hypersurface and $\deg (F')\in S=\{(a,b),(b+m-n,a)\}$. More yet, we have $\#S=1$ wherever:
\begin{itemize}
    \item $a=b$ or
    \item $m<n$ and $0<a-b \neq m$.
\end{itemize}
\end{prop}
\begin{proof}
  Let $Y=V(F)$ be a hypersurface of $X$ such that $\deg F=(a,b)$. The degree of $p_{Y'}$ implies that $Y'$ is also a 
  hypersurface, say of bidegree $(a',b')$. Comparing the leading coefficients of $p_Y$ and $p_{Y'}$, we obtain 
  $ma + nb = ma' + nb'$. Since $p_Y(t) = p_{Y'}(t)$, it follows that
    \[
    h^0(X,\OO(t-a,t-b)) = h^0(X,\OO(t-a',t-b')),
    \]
    which is equivalent to
    \[
    \binom{t - a + m}{m} \binom{t - b + n}{n} = \binom{t - a' + m}{m} \binom{t - b' + n}{n},
    \]
or else
\begin{equation}{\label{eq123}}
    \begin{matrix}
{(t-a+m)\cdots(t-a+1)(t-b+n)\cdots(t-b+1)}\frac{1}{m!n!} & =  \\
{(t-a'+m)\cdots(t-a'+1)(t-b'+n)\cdots(t-b'+1)}\frac{1}{m!n!}. &
\end{matrix}
\end{equation}

   Assume without loss of generality that \( a \geq b \). Then, the largest root of the left-hand side of the 
   Equation \eqref{eq123} is \( a - 1 \). If  \( a' \geq b' \), then the largest root of the right-hand side 
   is \( a - 1 \) which implies \( a = a' \) and from \( ma + nb = ma' + nb' \), we deduce \( b' = b \). %

  Suppose $a'<b'$, then \( a = b' \)  . The smallest root of the left side of the Equation \eqref{eq123} is $b-n$ while the 
  smallest root of the right-hand side of the Equation \eqref{eq123} is \( \min(a' - m, b' - n) \). This implies 
  \( b = b'+n-m \) or \( b = b' \). The last equality would force $b'=a'=a$ and from Equation \eqref{eq123}, we would have $b'=a'$. Thus $(a',b')=(b+m-n,a)$.

If $a=b$ is clear that $\#S=1$. Hence, suppose $m<n$  and  $0<a-b \neq m$. If $(a',b')=(b+m-n,a)$, then
$$ma+nb= ma'+nb'= m(b-n+m)+na=mb-mn+m^2+na,$$
which is equivalent to 
$$nb-na+mn=mb-ma+m^2\Rightarrow n(b-a+m)=n(b-a+m),$$
which implies $m=n$ (absurd). Thus $\#S=1$.

\end{proof}


\begin{prop}\label{prop18}
Let $k \subset K$ be a field extension, and let $Y$ be a closed subscheme of $\mathbb{P}^{m} \times \mathbb{P}^{n}$. Then $Y$ is arithmetically Cohen-Macaulay if and only if $Y_K$ is ACM.
\end{prop}

\begin{proof}
Let $Y = \bigcap_{i=1}^{c} V(F_i)$, and let $Y_K = \bigcap_{i=1}^{c} V(F_{i,K})$.  

\noindent Suppose $Y$ is ACM. By the flatness of base change, $Y_K$ is also a complete intersection. Since the bidegrees are preserved under base change, $Y_K$ is projectively normal, hence ACM. Conversely, suppose $Y_K$ is ACM in $X_K$. Let $G_1, \dots, G_c$ be a regular sequence defining $Y_K$, ordered appropriately.\\ 
We proceed by induction on $0 \leq r \leq c$ to show that (after possibly modifying the regular sequence) there exist $F_i \in H^{0}(X, \mathcal{I}_{Y}(a_i,b_i))$ such that $G_i = F_{i,K}$. The base case ($r = 0$) is trivial.\\

  \noindent\textbf{Inductive step:} Assume the claim holds for some $r = j$. Let $Y_j = \bigcap_{i=1}^{j} V(F_i)$. Consider the inclusion of $k$-vector spaces:  
  \[
  H^{0}(X, \mathcal{I}_{Y_j}(a_{j+1},b_{j+1})) \subset H^{0}(X, \mathcal{I}_{Y}(a_{j+1},b_{j+1})).
  \]  
  Extending scalars to $K$, the base change theorem gives an inclusion of $K$-vector spaces:  
  \[
  H^{0}(X, \mathcal{I}_{Y_{j,K}}(a_{j+1},b_{j+1})) \subset H^{0}(X, \mathcal{I}_{Y_K}(a_{j+1},b_{j+1})).
  \]  
  Since $G_{j+1} \in H^{0}(X, \mathcal{I}_{Y_K}(a_{j+1},b_{j+1}))$ but $G_{j+1} \notin H^{0}(X, \mathcal{I}_{Y_{j,K}}(a_{j+1},b_{j+1}))$, the inclusion is strict. By flatness, the original inclusion over $k$ is also strict. Thus, there exists $F_{j+1} \in H^{0}(X, \mathcal{I}_{Y}(a_{j+1},b_{j+1}))$ such that  
  \[
  F_{j+1} \notin H^{0}(X, \mathcal{I}_{Y_j}(a_{j+1},b_{j+1})).
  \]  

  By Proposition \ref{prop15}, $F_{j+1,K} \in H^{0}(X, \mathcal{I}_{Y_K}(a_{j+1},b_{j+1}))$ can be written as  
  \[
  F_{j+1,K} = \sum_{i=1}^{c} G_i H_i,
  \]  
  where $H_i \in H^{0}(X_K, \mathcal{O}_{X_K}(a_{j+1}-a_i, b_{j+1}-b_i))$. Since $F_{j+1,K} \notin H^{0}(X, \mathcal{I}_{Y_{j,K}}(a_{j+1},b_{j+1}))$, there exists $l \geq j+1$ such that $H_l \neq 0$. This implies $(a_l, b_l) = (a_{j+1}, b_{j+1})$, and we may assume $l = j+1$. Thus, we can replace $G_{j+1}$ with $F_{j+1,K}$ to obtain a new regular sequence for $Y_K$.\\

Let $Y' = \bigcap_{i=1}^{c} V(F_i)$. Since $Y' \subset Y$ and $Y'_K = Y_K$, the flatness of $k \to K$ ensures $Y' = Y$. Finally, since an ACM subscheme is determined by its bidegree and the induction process preserves bidegrees (up to reordering), the result follows.
\end{proof}


\begin{prop}\label{pro21}
    Let $T$ be a Noetherian scheme and $Y \subset X_T$ a closed subscheme such that the projection $\pi: Y \to T$ is flat with ACM fibers. Then for all $d \in \mathbb{Z}$, $\pi_* \mathcal{O}_Y(d)$ is locally free and its formation commutes with arbitrary base change.

    Moreover, if the fibers have bidegrees $(a_1,b_1), \ldots, (a_c,b_c)$ and $(a_{c+1},b_{c+1})$ satisfies the ACM relations of these fibers, then $\pi_* \mathcal{O}_Y(a_{c+1},b_{c+1})$ is locally free and commutes with base change.
\end{prop}

\begin{proof}
    For every $t \in T$, we have the commutative diagram:
    \[
    \begin{tikzcd}
        \pi_*\mathcal{O}_{X_t}(a_{c+1},b_{c+1}) \otimes k(t) \arrow{r}{pr} \arrow{d} & H^0(X_t, \mathcal{O}_{X_t}(a_{c+1},b_{c+1})) \arrow{d} \\
        \pi_*\mathcal{O}_Y(a_{c+1},b_{c+1}) \otimes k(t) \arrow{r}{pr'} & H^0(Y_t, \mathcal{O}_{Y_t}(a_{c+1},b_{c+1}))
    \end{tikzcd}
    \]
    By Proposition~\ref{prop15} or Corollary~\ref{cor11}, the right vertical arrow is surjective. Consequently, from \cite[Proposition 12.11-a]{RH77}, the left vertical arrow is also surjective. The map $pr$ is surjective by definition of the polynomial, hence $pr'$ is surjective. Therefore, by \cite[Proposition 1.1.7]{OB12}, $\pi_* \mathcal{O}_Y(a_{c+1},b_{c+1})$ is locally free and commutes with base change.
\end{proof}

By adapting the arguments of Olivier in \cite[Proposition 2.1.12]{OB12} for the biprojective space, we provide a proof of the following result:  

\begin{prop}\label{prop20}
    Let $T$ be the spectrum of a DVR, and let $Y \subset X_T$ be a closed subscheme that is flat over $T$. 
    If the special fiber of $Y$ is an ACM complete intersection, then the generic fiber is also an ACM 
    complete intersection. Moreover, any defining equation of the special fiber lifts to a defining 
    equation of $Y$.
\end{prop}

\begin{proof}

Let $T$ be a DVR with special point $s = \operatorname{Spec}(k)$ and generic point $\eta = \operatorname{Spec}(K)$. Consider a flat family $Z \subset X_T$ such that the special fiber $Z_s$ is an ACM complete intersection in $X_s$, defined by a regular sequence $F_1, \dots, F_c$.

Since $X_T$ and $Z$ are flat over $T$, the exact sequence
\[
0 \to \mathcal{I}_{Z} \to \mathcal{O}_{X_T} \to \mathcal{O}_{Z} \to 0
\]
implies that $\mathcal{I}_{Z}$ is also flat over $T$. Moreover, restriction to the special fiber $Z_s$ preserves exactness.

By Corollary \ref{cor11}, we have $H^1(X_k, \mathcal{I}_{Z_s}(d)) = 0$ for all $d \in \mathbb{Z}$. Semi-continuity then ensures that $H^1(X_K, \mathcal{I}_{Z_\eta}(d)) = 0$.

From Proposition \ref{pro21}, the sheaf $\pi_* \mathcal{I}_Z(d)$ is locally free, and its formation commutes with base change. Since $T$ is local, $\pi_* \mathcal{I}_Z(d)$ is in fact free. Thus, any section $F \in H^0(X_k, \mathcal{I}_{Z_s}(d)) = (\pi_* \mathcal{I}_Z(d))_s$ lifts to a global section $\widetilde{F} \in H^0(X_T, \mathcal{I}_Z(d))$, proving the second claim.

We now lift the defining equations $F_1, \dots, F_c$ of $Z_s$ to $\widetilde{F}_1, \dots, \widetilde{F}_c \in H^0(X_T, \mathcal{I}_Z)$, defining a subscheme $Z' = V(\widetilde{F}_1, \dots, \widetilde{F}_c)$ with $Z \subset Z'$ and $Z_s = Z'_s$. By semi-continuity, the generic fiber $Z'_\eta$ is a complete intersection cut out by the regular sequence $\widetilde{F}_{1,\eta}, \dots, \widetilde{F}_{c,\eta}$.

Flatness ensures that $Z_s$ and $Z_\eta$ have the same Hilbert polynomial. By Proposition \ref{prop15}, $Z_s$ and $Z'_\eta$ also share the same Hilbert polynomial. Hence, $Z_\eta$ and $Z'_\eta$ must coincide, as their Hilbert polynomials agree and $Z_\eta \subset Z'_\eta$.

We conclude that $Z_\eta$ is a complete intersection. Since the bidegrees are preserved, then we have that the forms $F_i$ and $\widetilde{F}_{i,\eta}, i=1,\cdots, c$ have the same bidegree  and thus, $Z_\eta$ is ACM.
 
\end{proof}
\section{Hilbert Scheme of Canonical Complete Intersection Curves}



In this Section we construct the Hilbert scheme $\mathcal{H}$ of complete intersections in the biprojective space $X=\mathbb{P}^m\times\mathbb{P}^n$. To do so, we adapt the methods developed by Olivier in \cite{OB12} for $\mathbb{P}^N$. 

Let $\Hilb_r$ denote the Hilbert scheme of $\P^{r}$, and let $\Hilb_r^{p(t)} \subset \Hilb_r$ be the subscheme corresponding to subschemes with Hilbert polynomial $p(t)$. 
For any locally Noetherian separated $k$-scheme $S$, define the \textbf{Hilbert functor}
\[
\h_X^{p(t)}(S) = 
\left\{
\begin{aligned}
&\text{Flat families } \mathcal{X} \subset X \times S \text{ of closed subschemes,} \\
&\text{parametrized by } S, \text{ with fibers having Hilbert polynomial } p(t).
\end{aligned}
\right\}
\]
Since flatness is preserved under base change, this defines a contravariant functor
\[
\h_X^{p(t)} \colon (\text{\underline{Sch}}/k)^{\text{op}} \to \text{\underline{Set}},
\]
where ${\underline{\rm Sch}}/k$ denotes the category of locally Noetherian separated $k$-schemes. It is well known 
that $\h_X^{p(t)}$ is representable by a closed subscheme of $\Hilb_r^{p(t)}$, which  will simply denoted by $\Hilb_X^{p(t)}$. 

The key difference between the standard case and the biprojective case is that not all complete intersections are ACM. Even if two curves share the same Hilbert polynomial and are complete intersections, neither is guaranteed to be ACM.

\begin{Example}
    \textup{Let \( Y = V(F_1, F_2) \) be a complete intersection in \( X = \mathbb{P}^{1} \times \mathbb{P}^{2} \), where \( F_1 \) and \( F_2 \) have bidegrees \( (2,2) \) and \( (1,2) \), respectively. Then \( Y \) has Hilbert polynomial \( p_Y(t) = 10t - 5 \). Now consider another complete intersection \( Y' \subset X \) defined by forms of bidegrees \( (3,2) \) and \( (0,2) \). Although \( p_Y = p_{Y'} \) and both curves are cut out by regular sequences, only \( Y \) satisfies the conditions in Observation~\ref{ord}.}
\end{Example}

For this reason, we restrict our study to \emph{canonical curves}: connected, non-bidegenerate curves whose dualizing sheaf is the restriction to a hyperplane section.

\begin{prop}{\label{proamp}}
Let $\CC$ be a complete intersection in $X$ such that $\omega_\CC = \OO_X(1)|_\CC$. If $\CC$ is cut out by ample divisors, then $\CC$ is ACM, in particular it is canonical.                       
\end{prop}

\begin{proof}
 
Let $\mathcal{C} = \bigcap_{i=1}^{m+n-1} V(F_i)$ be a curve cut out by ample divisors, with $\omega_{\mathcal{C}} \cong \mathcal{O}_{\mathcal{C}}(1)$ and $\deg(F_i) = (a_i, b_i)$.  Consider the case where $m = n = 1$. The only possible bidegree for $\CC$ is $(3,3)$, which is trivial. Thus, we may assume $n > 1$. 

We first show that $\mathcal{C}$ is generated by a regular sequence. Suppose, for contradiction, that there exists a subset of indices $i_1 < \cdots < i_r$ such that  
\[
\sum_{j=1}^r (a_{i_j} - b_{i_j}) > m.
\]
From the degree conditions, we have  
\[
m + 2 = \sum_{i=1}^{m+n-1} a_i \geq \sum_{j=1}^r a_{i_j} > m + \sum_{j=1}^r b_{i_j} \geq m + r.
\]
If $r > 1$, this implies $m + 2 > m + r \geq m + 2$, a contradiction. If $r = 1$, then  
\[
m + 2 > a_i > b_i + m,
\]
another contradiction. Thus $\sum_{j=1}^r (a_{i_j} - b_{i_j}) < m + 1$. Similarly, $\sum_{j=1}^r (b_{i_j} - a_{i_j}) < n + 1$. By Proposition~\ref{P10}, this is equivalent to $\mathcal{C}$ being defined by a regular sequence.  

Next, we verify that the bidegree ordering satisfies the conditions of Remark~\ref{ord}. Suppose, for contradiction, that there exists a subset of indices $i_1 < \cdots < i_\alpha$ and an index $\gamma \notin \{i_1, \dots, i_\alpha\}$ such that  
\[
\sum_{j=1}^\alpha a_{i_j} \geq a_\gamma + m + 1.
\]
Summing over all degrees, we obtain  
\[
m + 2 = \sum_{i=1}^{m+n-1} a_i \geq 2a_\gamma + m + 1.
\]
Since $a_\gamma \geq 1$, this implies $m + 2 \geq m + 3$, which is false. Thus, no such subset exists, and 
the claim follows. We conclude from Corollary \ref{cor12} that $\CC$ is canonical.
\end{proof}

Fix \( p(t) = (2g - 2)t + 1 - g \) and assume \( g \leq (m + 1)(n + 1) \). Let \( F \) denote the contravariant functor 
from the category of schemes to the category of sets defined as follows: for a scheme \( S \),

\[
F(S) = \left\{
\begin{aligned}
&\text{flat families } \mathcal{X} \subset X \times S \text{ of closed subschemes of X, parametrized  by $S$,  whose }  \\
&\text{  }  \text{fibers are canonical complete intersection curves cut out by ample divisors}
\end{aligned}
\right\}.
\]



Let us consider
$(\alpha_1, \beta_1), \dots, (\alpha_s, \beta_s)$ be the distinct bidegrees 
among the pairs $(a_i, b_i)$, grouped as Equation \eqref{o1}. First, we construct a scheme $\h$ equipped with a universal family $\mathcal{X}$. 
To do this, we inductively define, for each $0 \leq r \leq s$, a flat family
\[
p_r \colon \mathcal{X}_r \hookrightarrow X \times {\h}_r \to {\h}_r,
\]
where ${\h}_r$ is an integral, smooth $\mathbb{Z}$-scheme, and the fibers of $p_r$ are ACM subschemes cut out by $m_t$ equations of bidegree $(\alpha_t, \beta_t)$ for $1 \leq t \leq r$.


If \( r = 0 \), we set \(\mathcal{H}_0 = \Spec(\mathbb{Z})\) and \(\mathcal{X}_0 = X\). Proceeding by induction, we assume that \(\mathcal{X}_{r-1}\) and \(\mathcal{H}_{r-1}\) are constructible. By Proposition~\ref{pro21}, the sheaf \(\mathcal{E}_r = p_{r-1,*}\mathcal{O}_{\mathcal{X}_{r-1}}(\alpha_r, \beta_r)\) is locally free and commutes with base change. We may thus define the Grassmannian bundle \(G_r = \mathrm{Gr}(m_r, \mathcal{E}_r^\vee)\) with its natural projection \(\pi_r: G_r \to \mathcal{H}_{r-1}\). This construction yields the following commutative diagram:

\[
\begin{tikzcd}
\pi_{r}^{*}\mathcal{X}_{r-1} \arrow{r}{} \arrow[swap]{d}{p_r} & \mathcal{X}_{r-1} \arrow{d}{p_{r-1}} \\
G_r \arrow{r}{\pi_r} & \mathcal{H}_{r-1}
\end{tikzcd}
\]

On \(G_r\), consider the inclusion of the tautological sheaf \(\mathcal{F}_r \hookrightarrow \pi_r^{*}\mathcal{E}_r\). After a change of basis for \(\pi_r\), this injection may be expressed as
\[ \mathcal{F}_r \hookrightarrow p_{r,*}\mathcal{O}_{\pi_r^{*}\mathcal{X}_{r-1}}(\alpha_r, \beta_r). \]
Pulling back to \(\pi_r^{*}\mathcal{X}_{r-1}\) and applying adjunction, we obtain a morphism of vector bundles
\[ p_r^{*}\mathcal{F}_r \to \mathcal{O}_{\pi_r^{*}\mathcal{X}_{r-1}}(\alpha_r, \beta_r). \]
The zero locus of this morphism defines a closed subscheme \(Z_r \subset \pi_r^{*}\mathcal{X}_{r-1}\).

This step mirrors Olivier's method in \cite{OB12}, but replaces homogeneous forms on \(\mathbb{P}^N\) by bihomogeneous forms on \(X\). For a point \(y \in G_r\) with image \(x = \pi_r(y)\), let \(V\) be the corresponding \(m_r\)-dimensional subspace of
\[ \mathcal{E}_{r,x} = H^{0}\big(\mathcal{X}_{r-1,x}, \mathcal{O}_{\mathcal{X}_{r-1,x}}(\alpha_r, \beta_r)\big). \]
By construction, the fiber \(Z_{r,y}\) coincides with the subscheme of \(\mathcal{X}_{r-1}\) cut out by \(\{F = 0 \mid F \in V\}\), ensuring that \(Z_{r,y}\) has codimension \(\leq m_r\) in \(\mathcal{X}_{r-1}\).

Define \(\mathcal{H}_r \subset G_r\) as the open subscheme where \(Z_{r,y}\) attains minimal dimension. Here, \(Z_{r,y}\) is a complete intersection defined by \(m_i\) equations of bidegree \((\alpha_i, \beta_i)\) for \(1 \leq i \leq r\). We then take \(\mathcal{X}_r\) to be the restriction of \(Z_r\) to \(\pi_r^{-1}(\mathcal{H}_r)\), verifying that \(\mathcal{X}_r\) and \(\mathcal{H}_r\) satisfy all required hypotheses. We gonna call $\h$, the Hilbert scheme of complete intersection.

\begin{Lemma}{\label{bij}}
For any field $k$, $\Phi(k): \h(k) \to F(k)$ is a bijection.
\end{Lemma}

\begin{proof}
Let $k$ be a field. The construction from the previous paragraph shows that a $k$-point of $\mathfrak{h}$ consists of the following data for each $0 \leq r \leq s$:
\begin{itemize}
    \item A complete intersection $Y_{r} \subset X_{k}$ of bidegrees $(\alpha_1, \beta_1), \ldots, (\alpha_{\mu_r}, \beta_{\mu_r})$, and
    \item A subspace $V_{r} \subset H^{0}(Y_{r-1}, \mathcal{O}(\alpha_r, \beta_r))$ of dimension $m_r$,
\end{itemize}
such that $Y_{r}$ is the zero locus in $Y_{r-1}$ of the elements of $V_{r}$. The map $\Phi(k)$ sends this element of ${\h}(k)$ to the complete intersection $Z = Y_{s}$.

\textbf{Injectivity of $\Phi(k)$.}
Suppose $Z$ is a complete intersection over $k$. By Proposition~\ref{prop15}, the subschemes $Y_{r} \subset Z$ are uniquely determined. Moreover, each $V_{r}$ must lie in the kernel of the restriction map
\[ H^{0}(Y_{r-1}, \mathcal{O}(\alpha_r, \beta_r)) \to H^{0}(Z, \mathcal{O}(\alpha_r, \beta_r)), \]
and hence $V_{r}$ equals this kernel. Thus, $Z$ has at most one preimage under $\Phi(k)$.

\textbf{Surjectivity of $\Phi(k)$.}
Let $Z$ be a complete intersection over $k$, defined by a global regular sequence $F_1, \ldots, F_c$. For each $r$, set:
\begin{itemize}
    \item $Y_{r} = \bigcap_{i=1}^{\mu_r} V(F_i) $,
    \item $V_{r} = \langle F_{\mu_{r-1}+1}, \ldots, F_{\mu_r} \rangle$.
\end{itemize}
Then $Y_{r}$ is the vanishing locus of $V_{r}$ in $Y_{r-1}$. The dimension of $V_{r}$ is $m_r$, since any linear dependence among $F_{\mu_{r-1}+1}, \ldots, F_{\mu_r}$ in $H^{0}(Y_{r-1}, \mathcal{O}(\alpha_r, \beta_r))$ would force $Y_{r}$ to have codimension  less than $ m_r$ in $Y_{r-1}$.

\end{proof}

We are interested in canonical curves in the biprojective $X$ and for this, we will make the following considerations: we consider a 
connected, non-bidegenerated curve of genus $g$ cut out by ample divisors, such that its canonical sheaf $\omega_{\CC}$ is the restriction 
of a hyperplane section. More precisely, the 
connectedness of $\CC$ is equivalent to $h^{0}(\CC,\OO_{\CC}) = 1$ or 
$h^1(\mathcal{I}_{\CC,X}) = 0$ (the minimal possible dimension) while the non-bidegeneracy of the curve means that $\CC$ is not 
contained in any projection, i.e. the natural restriction map
\[
H^{0}\big(X,\OO_X(a,b)\big) \longrightarrow H^{0}\big(\CC,\OO_\CC(a,b)\big)
\]
is an isomorphism for $(a,b) \in \{(1,0), (0,1)\}$. This isomorphism imposes additional constraints on the ideal sheaf:
\[
h^1\big(\mathcal{I}_{\CC,X}(a,b)\big) = h^0\big(\mathcal{I}_{\CC,X}(a,b)\big) = 0 \quad \text{for} \quad (a,b) \in \{(1,0), (0,1)\}.
\]

Since we want $g\leq(m+1)(n+1)$, we impose that the natural restriction map
$H^{0}\big(X,\OO_X(1)\big) \longrightarrow H^{0}\big(\CC,\OO_\CC(1)\big)$
to be surjective, i.e. $h^{1}(\mathcal{I}_{\CC,X}(1))=0$. Consequently, a canonical curve corresponds to a point in the open subset
\[
\mathrm{K}^{1} = \left\{ y \in \Hilb^{p(t)}_X \;\Big|\; h^{1}(\mathcal{I}_{y}(1))=h^{i}(\mathcal{I}_{y}) = h^{i}(\mathcal{I}_{y}(1,0)) = h^{i}(\mathcal{I}_{y}(0,1)) = 0, \;   i \leq 1 \right\},
\]
where openness follows from semicontinuity. 

To ensure $\CC$ is generated by ample divisors, it suffices to require
\[
h^{0}(\mathcal{I}_{\CC,X}(a,0)) = h^{0}(\mathcal{I}_{\CC,X}(0,b)) = 0 \quad \forall a \in \{1, \dots, m+1\} \quad \text{and} \quad   \forall b \in \{1, \dots, n+1\}.
\]
Define
\[
\mathrm{K}^0 = \left\{ y \in \mathrm{K}^1 \;\Big|\; h^{0}(\mathcal{I}_y(a,0)) = h^{0}(\mathcal{I}_y(0,b)) = 0, \text{ for all }  1 \leq a \leq m+1,\, 1 \leq b \leq n+1 \right\}
\]
By semicontinuity, $\mathrm{K}^0$ is open in $\mathrm{K}^1$.

Finally, $\CC$ has the additional property that the restriction $\OO_X(1)|_\CC$ is isomorphic to $\omega_\CC$. 
Since the dualizing sheaf of a complete intersection depends only on the bidegree (by Proposition~\ref{K}), 
we may assume $\CC$ is smooth. In this case, the isomorphism $\OO_X(1)|_\CC \cong \omega_\CC$ is equivalent 
to $h^1(\OO_\CC(1)) > 0$, which follows immediately from the Riemann-Roch Theorem and the equality of degrees of the two sheaves.

We set
\[
\mathrm{K} = \left\{ y \in \mathrm{K}^0 \;\Big|\; h^{2}(\mathcal{I}_y(1)) > 0 \right\}.
\]
By semicontinuity, $\mathrm{K}$ is closed in $\mathrm{K}^0$.

 \begin{prop}\label{prop23}
The functor $F$ is representable by an open set of $\operatorname{K}$.
 \end{prop}

 \begin{proof}
     Let $U \subset \operatorname{K}$ denote the set of points $x$ for which the corresponding subscheme in $X_{\kappa(x)}$ is a complete intersection over $\kappa(x)$. We claim that $U$ is an open set of $\operatorname{K}$ which represents $F$. For this we need to show that $U$ is constructible and stable under generization.

     Let $x, y \in \operatorname{K}$ such that $y \in \overline{\{{x}\}}$ and $y \in U$. There exists a trait (spectrum of a DVR) $T$ with closed point $s$ and generic point $\eta$, and a morphism $f: T \to \operatorname{K}$ such that $f(s) = y$ and $f(\eta) = x$. Denote by $\mathcal{Y}$ the universal family of $\operatorname{Hilb}_{X}^{p(t)}$, pulling back  of $\mathcal{Y}\cap \operatorname{K}$ over $T$, we obtain a subscheme $Z \subset \operatorname{K} \times T$ flat over $T$. Since $y \in U$, we can apply Proposition \ref{prop20} to it: $Z_{\eta}$ is a complete intersection. Since $Z_{\eta}$ is obtained by scalar extension from the subscheme of $X_{\kappa(x)}$ associated to $x$, Proposition \ref{prop18} shows that $x \in U$. We have shown that $U$ is stable under generization.

     Consider the morphism of functors $i \circ \Phi: \mathcal{H} \to \operatorname{Hilb}^{p(t)}_{X}$; this is a morphism of schemes. If $x \in U$, the surjectivity in Lemma \ref{bij} shows that $x$ is in the image of $i \circ \Phi$. Conversely, let $x$ be a point in the image of $i \circ \Phi$ and $y$ a preimage of $x$. Then $\mathcal{X}_{y}$ is obtained by scalar extension from the subscheme of $X_{\kappa(x)}$ associated to $x$. By Proposition \ref{prop18}, this shows $x \in U$. Thus, $U$ is the set-theoretic image of $i \circ \Phi$. By Chevalley's theorem, $U$ is therefore constructible. This concludes that $U$ is an open set of $\operatorname{K}$.

     From Proposition \ref{prop15}, we know that all complete intersections have the same Hilbert polynomial $p(t)$, so there 
     is an obvious monomorphism of functors $i: F \to \operatorname{Hilb}^{p(t)}_{X}$. By Proposition \ref{prop18}, the morphism 
     of functors $i$ factors as $i: F \to U$. The definition of the morphism of functors $i$ shows that it is a 
     monomorphism. It is immediate that $i: F \to U$ is an epimorphism. Consequently, $i: F \to U$ is an 
     isomorphism of functors: $F$ is indeed representable by $U$.
 \end{proof}
 \begin{cor}
The morphism $\Phi:\h \to U$ induced by the family $pr: \mathcal{X} \to \h$ is an isomorphism. 
 \end{cor}
 \begin{proof}
   We begin by constructing a morphism $\Psi \colon U \to {\h}$ that will later be shown to be the inverse of $\Phi$. Let $\mathrm{pr}_{U} \colon \mathcal{X} \to U$ denote the universal family. Retaining the previous notations, we inductively construct, for each $0 \leq r \leq s$, a morphism $\Psi_{r} \colon U \to {\h}_{r}$ such that $\mathcal{X}$ is a closed subscheme of $\Psi_{r}^{*}\mathcal{X}_{r}$ inside $X \times U$.

For $r=0$, we take $\Psi_{0}: U \to \h_{0} = \operatorname{Spec}(\mathbb{Z})$ as the structural morphism. Assume $\Psi_{r-1}$ is constructed. By Proposition~\ref{pro21}, the sheaves
\[
\mathrm{pr}_{U,*}\mathcal{O}_{\Psi_{r-1}^{*}\mathcal{X}_{r-1}}(\alpha_r, \beta_r) \quad \text{and} \quad \mathrm{pr}_{U,*}\mathcal{O}_{\mathcal{X}}(\alpha_r, \beta_r)
\]
are vector bundles on $U$ whose formation commutes with base change. Proposition~\ref{prop15} implies that for any $u \in U$, the restriction map
\[
H^{0}\big((\Psi_{r-1}^{*}\mathcal{X}_{r-1})_{u}, \mathcal{O}(\alpha_r, \beta_r)\big) \to H^{0}\big(\mathcal{X}_{u}, \mathcal{O}(\alpha_r, \beta_r)\big)
\]
is surjective with a kernel of dimension $m_r$. Thus, the bundle map
\[
\mathrm{pr}_{U,*}\mathcal{O}_{\Psi_{r-1}^{*}\mathcal{X}_{r-1}}(\alpha_r, \beta_r) \to \mathrm{pr}_{U,*}\mathcal{O}_{\mathcal{X}}(\alpha_r, \beta_r)
\]
is surjective, and its kernel $N$ is a vector bundle of rank $m_r$. By the definition of the Grassmannian $G_r$, this induces a morphism $\Psi_{r} \colon U \to G_r$. Let $F_1, \dots, F_c$ be a global regular sequence defining $\mathcal{X}_{u}$. Then:
\begin{itemize}
    \item $(\Psi_{r-1}^{*}\mathcal{X}_{r-1})_{u} = \bigcap_{i=1}^{\mu_{r-1}}V (F_i)$,
    \item $N_u$ is generated by $F_{\mu_{r-1}+1}, \dots, F_{\mu_r}$,
    \item $(\Psi_{r-1}^{*}{Z}_{r})_{u} = \bigcap_{i=1}^{\mu_r} V(F_i)$ is a complete intersection.
\end{itemize}

This shows that $\Psi_{r}$ factors through $\h_{r}$, yielding $\Psi_{r} \colon U \to {\h}_{r}$. By construction, $\mathcal{X} \subset \Psi_{r}^{*}\mathcal{X}_{r}$, completing the induction. Set $\Psi = \Psi_s$.

We will verify that $\Psi$ is the inverse of $\Phi$. First, we show that $\Phi \circ \Psi: U \to U$ is the identity. Given the functor represented by $U$, it suffices to verify that the two subschemes $\mathcal{Y}$ and $\Psi^{*}\mathcal{X}$ of $X \times U$ coincide. By construction, we know the inclusion $\mathcal{Y} \subset \Psi^{*}\mathcal{X}$; we then denote by $\mathcal{N}$ the ideal sheaf defining $\mathcal{Y}$ in $\Psi^{*}\mathcal{X}$, so that we have a short exact sequence of sheaves on $X \times U$:

$$0 \to \mathcal{N} \to \mathcal{O}_{\Psi^{*}\mathcal{X}} \to \mathcal{O}_{\mathcal{Y}} \to 0.$$

By the flatness of $\mathcal{Y}$ over $U$, the right-hand sheaf is flat over $U$, so by base change to a point $u \in U$, we obtain the following short exact sequence of sheaves on $X_{u}$:

\[
0 \to \mathcal{N}_{u} \to \mathcal{O}_{(\Psi^{*}\mathcal{X})_{u}} \to \mathcal{O}_{\mathcal{Y}_{u}} \to 0.
\]

By Proposition \ref{prop15}, the two right-hand sheaves have the same Hilbert polynomial. By additivity, the Hilbert polynomial of $\mathcal{N}_{u}$ is zero, hence $\mathcal{N}_{u} = 0$. Since this holds for all $u \in U$, Nakayama's lemma shows that $\mathcal{N} = 0$. Thus, $\Psi^{*}\mathcal{X} = \mathcal{Y}$, as desired.

Let $k$ be a field. Since $\Phi \circ \Psi = \text{Id}$ and, by Lemma \ref{bij}, $\Phi(k)$ is a bijection, $\Psi(k)$ is the inverse of $\Phi(k)$. Thus, $\Psi \circ \Phi$ induces the identity on $k$-points for any field $k$. The two morphisms of schemes $\Psi \circ \Phi: \h \to \h$ and $\text{Id}: \h \to \h$ therefore induce the same maps on $k$-points for any field $k$. Since $\h$ is reduced, these two morphisms coincide by \cite[Lemma 2.2.6]{OB12}.

We have shown that $\Psi$ is the inverse of $\Phi$, so $\Phi$ is an isomorphism.
 \end{proof}




\section{Moduli of Canonical Complete Intersection}

The goal of this section is to provide an application to the construction of the moduli space of canonical curves. More precisely, we 
focus on curves studied by Mukai in \cite{MUK95}.

Let $\mathcal{H}$ denote the Hilbert scheme parametrizing (not necessarily smooth) curves of genus~$7$, as in Example~\ref{ex13}. 
Observe that the trigonal case is the only locus of non-canonical curves; nevertheless, the construction 
of $\mathcal{H}$ is entirely determined by the bidegrees. 

\begin{Example}{\label{ult}} 
\textup{Define:
\[
V(r,s) := H^0(X, \OO_X(r,s))^\vee, \quad \P(r,s) := \P\left(H^0(X, \OO_X(r,s))^\vee\right).
\]
Note that, by construction, the Hilbert scheme of complete intersections is an open subscheme of a Grassmannian bundle; for simplicity, we will not distinguish them. 
\begin{enumerate}[label=(\arabic*)]
    \item \textbf{Case}{ $(a_1,b_1)=(1,1)$ and $(a_2,b_2)=(3,3)$ on $\P^1 \times \P^2$:}
    There exists a projective bundle $\pi: \h \to \P(1,1)$ with fibers
    \[
    \pi^{-1}(F) = \P\left(\frac{V(3,3)}{F \cdot V(2,2)}\right), \quad \dim \pi^{-1}(F) = 21,
    \]
    implying $\dim \h= 26$.
    \item \textbf{Case}{ $(a_1,b_1)=(1,1)$ and $(a_2,b_2)=(a_3,b_3)=(1,2)$ on $\P^1 \times \P^3$:}
    We obtain a Grassmannian bundle $\pi: \h \to \P(1,1)$ with fibers
    \[
    \pi^{-1}(F) = G\left(2, \frac{V(1,2)}{F \cdot V(0,1)}\right), \quad \dim \pi^{-1}(F) = 28,
    \]
    yielding $\dim \h = 35$.
\item \textbf{Case}{ $(a_1,b_1)=(a_2,b_2)=(1,1)$ and $(a_3,b_3)=(2,2)$ on $\P^2 \times \P^2$:}
    we have a projective bundle $\pi: \h \to G(2, V(1,1))$ with fibers
    \[
    \pi^{-1}(F,G) = \P\left(\frac{V(2,2)}{F \cdot V(1,1) + G \cdot V(1,1)}\right).
    \]
    Since $\dim(F \cdot V(1,1) + G \cdot V(1,1)) = 17$, we have $\dim \pi^{-1}(F,G) = 17$, and thus $\dim \h = 32$.
\end{enumerate}}
\end{Example}

\begin{Example}\label{g=8}
\textup{The only canonical curve of genus $8$ in Example~\ref{ex13}, cut out by ample divisors, is the pentagonal curve equipped with a non-selfadjoint linear system \(\mathfrak{g}^2_7\). Consequently, 
\[
(a_1, b_1) = (1, 1), \quad (a_2, b_2) = (1, 2), \quad (a_3, b_3) = (2, 1).
\]
This yields a chain of Grassmannians:
\[
\mathfrak{\h} = \mathfrak{\h}_2 \longrightarrow \mathfrak{\h}_1 \longrightarrow \mathfrak{\h}_0 = \mathbb{P}(1, 1) = \mathbb{P}^8,
\]
where:
\begin{itemize}
    \item \(\mathfrak{\h}_1\) is a projective bundle over \(\mathbb{P}^8\) with \(14\)-dimensional fibers.
    \item \(\mathfrak{\h}_2\) is a projective bundle over \(\mathfrak{\h}_1\) with \(14\)-dimensional fibers.
\end{itemize}
Therefore, \(\dim \mathfrak{\h}_1 = 8 + 14 = 22\) and \(\dim \mathfrak{\h}_2 = 22 + 14 = 36\).}
\end{Example}
Throughout this section, the group $G$ denotes either $\rm PGL(m+1) \times \rm PGL(n+1)$ or $\rm SL(m+1) \times \rm SL(n+1)$. The action of $\rm G$ on $X$ induces an action on $\h$.

\begin{Lemma}\label{finite}
Let $Y = \bigcap_{i=1}^c V(F_i)$ be a smooth ACM complete intersection in $X$. Suppose
\[
\sum_{i=1}^c a_i - m - 1 = \sum_{i=1}^c b_i - n - 1:=d > 0.
\]
Then the stabilizer subgroup $\rm Stab(Y) \subset \rm G$ is finite.
\end{Lemma}

\begin{proof}
Since $\mathrm{Stab}(Y) \subset \mathrm{Aut}(Y)$, it suffices to show that $\mathrm{Aut}(Y)$ is finite. From 
\cite[Lemma~3.4]{MO67}, the Lie algebra of $\mathrm{Aut}(Y)$ is isomorphic to $H^0(Y, T_Y)$, where $T_Y$ 
is the tangent sheaf on $Y$. Thus, 
$$\dim \mathrm{Aut}(Y) = 0\Leftrightarrow H^0(Y, T_Y) = 0.$$
Now, by the ACM property, $\omega_Y \cong \mathcal{O}_Y(d)$ and $h^1(\mathcal{I}_{Y,X}(-d)) = 0$. Since $h^0(\mathcal{O}_X(-d)) = 0$, the sheaf $T_Y$ does not have non-zero global sections, yielding $H^0(Y, T_Y) = 0$ as required.
\end{proof}
Let \( X \subset \mathbb{P}^N \) denote the {Segre embedding}. Under the Kronecker product, the group \( G \) is identified with a 
subgroup of \(\operatorname{PGL}(N+1)\) or \(\operatorname{SL}(N+1)\). If \(\Delta\) denotes the singular locus 
of \(\operatorname{Hilb}_{\mathbb{P}^N}^{p(t)}\), then \(\Delta\) is \(\operatorname{PGL}(N+1)\)-invariant. Consequently, 
the open subset \(\mathcal{U} \subset \mathcal{H}\) parametrizing smooth varieties 
is \(G\)-invariant. From this argument and Lemma \ref{finite} we prove the following result. 
\begin{Theorem}
Let $\mathcal{U} \subset \h$ be the locus of smooth canonical curves. Then a good quotient $[\mathcal{U}/\rm G]$ exists.
\end{Theorem}

Let $\mathcal{M}:=[\mathcal{U}/\rm G]$ be the quotient stack which is the moduli space of the complete intersections in $X$. We can observe that
\begin{equation}\label{dimM}
\dim \mathcal{M} = \dim \mathcal{U} - \dim \rm G.
\end{equation}
In particular, if we apply the formula \eqref{dimM} to the three cases of the Example \ref{ult} we obtain
\begin{cor} Let $\mathcal{M}$ be the moduli space of canonical curves of genus 7 on $X=\mathbb{P}^m\times\mathbb{P}^n$ that are cut out by ample divisors. 
\begin{enumerate}
    \item  If $X=\mathbb{P}^1\times\mathbb{P}^2$ and the ample divisors having bidegrees $(a_1,b_1)=(1,1), (a_2,b_2)=(3,3)$, then $\dim \mathcal{M} = 15$;
    \item For $X=\mathbb{P}^1\times\mathbb{P}^3$ and the ample divisors having bidegrees $(a_1,b_1)=(1,1), (a_2,b_2)=(a_3,b_3)=(1,2)$, it follows that $\dim \mathcal{M} = 17$;
 \item For $X=\mathbb{P}^2\times\mathbb{P}^2$ and the ample divisors have bidegrees $(a_1,b_1)=(a_2,b_2)=(1,1), (a_3,b_3)=(1,2)$, it follows that $\dim \mathcal{M} = 16$.
\end{enumerate}
\end{cor}
Note that the dimensions of the three moduli spaces in the above corollary coincide with those computed by Mukai in \cite[Table1]{MUK95}. Unlike the case of genus 7, Mukai does not computer the dimension of the moduli space of curves of genus 8 that are complete intersections in $\mathbb{P}^2\times\mathbb{P}^2$. From Example \ref{g=8} and applying Equation \eqref{dimM}, we determine the dimension of the moduli space $\mathcal{M}$ of canonical genus 8 curves cut out by ample divisors of bidegrees $(a_1,b_1)=(1,1), (a_2,b_2)=(1,2)$ and $(a_3,b_3)=(2,1)$. More precisely,
$$\dim\mathcal{M}=\dim\mathcal{H}-\dim G=36-16=20.$$

\begin{Remark}\label{final}
Let $\mathcal{M}^{Mu}$ denote the moduli space as defined by Mukai in ~\cite[Table~1]{MUK95} and let $\mathrm{G}_7^2$ be the irreducible subset of the moduli space of smooth genus~$8$ curves such that every  $C \in \mathrm{G}_7^2$ admits a $\mathfrak{g}_7^2$. Then $\dim \mathrm{G}_7^2 = 20$ and the locus of curves possessing a {non-self-adjoint} linear system~$L$ is an {open} subset of~$\mathrm{G}_7^2$ (characterized by the vanishing of $h^1(  L^{\otimes 2})$), hence $\dim \mathcal{M}^{\mathrm{Mu}}=20$ and thus $\dim\mathcal{M}^{\mathrm{Mu}}=\dim\mathcal{M}$. 
\end{Remark}

The Remark \ref{final} suggests a connection between curves equipped with a fixed linear system and complete intersection curves on $X$. For example, restricting any projection of $X$ to a curve $[\CC] \in \mathcal{M}$ yields a complete linear system $\mathfrak g_d^{r}$, where $r \in \{m, n\}$.
\newpage

\vspace{1cm}

\textbf{Author Contributions} The authors have equal contribution.

\textbf{Data availability} Data sharing not applicable to this article as no data sets were generated or analyzed during the current study.
\section*{Declaration}
\textbf{Conflict of interest} The authors have no relevant financial or non-financial interests to disclose.

\newpage
\bibliographystyle{amsalpha}

\end{document}